\documentclass[13pts, reqno]{amsart}
\usepackage{epsfig}
\usepackage{color}
\usepackage{amssymb,amsmath,amsthm,amstext,amsfonts}
\usepackage{psfrag}
\usepackage{url}
\usepackage{epstopdf}
\usepackage{epic,eepic}
\usepackage{upgreek}
\usepackage{amssymb}
\usepackage{mathrsfs}

\usepackage[colorlinks=true, linkcolor=blue, urlcolor=red, citecolor=blue, hyperindex]{hyperref}

\usepackage[all]{xy}

\makeatletter \@addtoreset{equation}{section} \makeatother

\renewcommand\thetable{\thesection.\@arabic\c@table}

\theoremstyle{plain}

\newtheorem{maintheorem}{Theorem}
\newtheorem{theorem}{Theorem}[section]
\newtheorem{proposition}[theorem]{Proposition}

\newtheorem{lemma}[theorem]{Lemma}

\newtheorem{example}[theorem]{Example}

\newtheorem{maincorollary}{Corollary}

\newtheorem{remark}[theorem]{Remark}

\newcommand{\vep}{\varepsilon}

\newcommand{\supp}{\operatorname{supp}}

\newcommand{\cP}{\mathcal{P}}

\newcommand{\cE}{\mathcal{E}}

\newcommand{\cF}{\mathcal{F}}

\newcommand{\cU}{\mathcal{U}}

\makeatletter
\newcommand{\tpitchfork}{%
  \vbox{
    \baselineskip\z@skip
    \lineskip-.52ex
    \lineskiplimit\maxdimen
    \m@th
    \ialign{##\crcr\hidewidth\smash{$-$}\hidewidth\crcr$\pitchfork$\crcr}
  }%
}
\makeatother

\newcounter{main}

\title{A criterion for the triviality of the centralizer for vector fields and applications}

\author[W. Bonomo]{Wescley Bonomo}
\address{W. Bonomo, Universidade Federal da Bahia\\
Av. Ademar de Barros s/n, 40170-110 Salvador, Brazil}
\email{wescleybonomo@yahoo.com.br}

\author[Paulo Varandas]{Paulo Varandas}
\address{Paulo Varandas, Departamento de Matem\'atica, Universidade Federal da Bahia\\
Av. Ademar de Barros s/n, 40170-110 Salvador, Brazil }
\email{paulo.varandas@ufba.br}
\urladdr{http://www.pgmat.ufba.br/varandas}

\begin{document}

\begin{abstract}
In this paper we establish a criterion for the triviality of the $C^1$-centralizer for vector fields and flows.
In particular we deduce the triviality of the centralizer at 
homoclinic classes of
$C^r$ vector fields ($r\ge 1$). Furthermore,  we show that the set of flows whose $C^1$-centralizer is trivial include:
(i) $C^1$-generic volume preserving flows, (ii) $C^2$-generic Hamiltonian flows on a generic and full Lebesgue measure set of energy
levels, and (iii) $C^1$-open set of non-hyperbolic vector fields  (that admit a Lorenz attractor).
We also provide a criterion for the triviality of the $C^0$-centralizer of continuous flows.
\end{abstract}

\date{\today}
\maketitle

\section{Introduction and statement of the main results}\label{sec:statements}

\subsection{Introduction}

In this article we study the centralizer of flows and vector fields. 
The centralizer of a flow is the set of flows that commute with the original one. 
In \cite{Sm}, Smale proposed a problem on the triviality of the centralizer in dynamical systems, asking if 
`typical' dynamics (meaning open and dense, or Baire generic) would have trivial centralizers. 
In the case of discrete time dynamical systems there have been substantial contributions towards 
both an affirmative solution to the problem raised by Smale and also to applications that exploit the relation of the centralizer 
with differentiability of conjugacies, embedding of diffeomorphisms as time-1 maps of
flows, rigidity of measurable $\mathbb Z^d$-actions on tori 
and  characterization of conjugacies for structurally stable diffeomorphisms, 
just to mention a few
(see e.g. \cite{BCW2, Bu04, Ko70, Fi09, Katok, Palis, PY89, Ro93, RV16, Yoccoz} and references therein).

The centralizer of continuous time dynamics has been less studied. Indeed, in the continuous time setting 
the problem of the centralizer has been considered for Anosov flows \cite{JH91}, $C^\infty$-Axiom A flows with the strong transversality condition~\cite{Sad2}, and for Komuro-expansive flows, singular-hyperbolic attractors and 
expansive $\mathbb R^d$-actions \cite{BRV,BoVa, oka}. In most cases, the ingredients in the proofs of the previous results
involve either a strong form of hyperbolicity and/or $C^\infty$-smoothness of the dynamics in order to use linearization of the dynamics in a neighborhood of a periodic or singular point.

In order to describe a larger class of vector fields (and flows) we establish a criterion 
for the triviality of the $C^1$-centralizer of vector fields based on the denseness of hyperbolic periodic orbits 
(Theorem~\ref{thm:criteriumH}). 
In particular the triviality of the centralizer can be deduced whenever there exists a dense set of hyperbolic periodic orbits, a condition that is satisfied by all homoclinic classes of $C^r$-vector fields in view of Birkhoff-Smale's theorem $(r\ge 1)$ (Corollary~\ref{cor:homoclinic}). 
We use this fact to show that $C^1$-generic volume preserving flows have a trivial centralizer
(Corollary~\ref{ks}), obtaining the counterpart of \cite{BCW2} for continuous time dynamics, 
and that  Lorenz attractors have trivial centralizer 
(cf. Corollary~\ref{cor:Lorenz}).
In the case of Hamiltonian dynamics the relation between symmetries and elements of the centralizer of a
Hamiltonian flow is more evident. Indeed, it follows from Noether's theorem that there exists a bijection
between conserved quantities, modulo constants, and the set of infinitesimal symmetries (see Subsection~\ref{subsec:Hamilton} for more details).  First results on commuting Hamiltonians were obtained in~\cite{Atiyah}, which described almost-periodic commuting Hamiltonian vector fields. 
Moreover, as the Hamiltonian is always a first integral for the 
Hamiltonian vector field it is natural to ask whether typical Hamiltonians have other independent first integrals.
Arnold-Liouville's integrability theorem assures that the existence of $n$ independent commuting integrals 
for a $C^2$-Hamiltonian $H$ implies the integrability of the Hamiltonian vector field and that the restriction of the flow to each invariant torus consists of a translation flow (see e.g.~\cite{Arnold}). However, $C^2$-generic Hamiltonians 
are transitive on each connected component of generic energy levels (\cite{BFRV}) and, consequently, continuous first integrals are constant on the corresponding level sets.
We use this fact here to prove 
that the centralizer of $C^2$-generic Hamiltonians on compact symplectic manifolds is trivial on generic level sets (we refer the reader to Corollary~\ref{hamil1} for the precise statement).

The problem of the centralizer for flows and vector fields differs substantially from the discrete-time setting. 
For instance, the statement of the criterion for triviality established here for vector fields is 
false even for Anosov diffeomorphisms (see Remark~\ref{rmkcrit}).
In rough terms, the first part of the strategy to the proof of Theorem~\ref{thm:criteriumH} is to show that the closure of the set of 
hyperbolic critical elements is preserved by all flows in the $C^1$-centralizer of a flow $(X_t)_t$.
The second part of the argument is to use transitivity to show that every element in the centralizer of $(X_t)_t$ is a linear 
(actually constant) reparametrization of $(X_t)_t$. One should mention that if one was interested to
describe a Baire generic subset of dynamical systems the  transitivity assumption could be removed. 
Indeed, as proposed by R. Thom \cite{Thom}, for every $1\le r\le \infty$, $C^r$-generic vector fields do not admit 
non-trivial first integrals (see e.g. \cite{mane}).

\subsection{Preliminaries}\label{sec:prelim}

In this subsection we recall some necessary definitions and set some notation.

\subsubsection{Non-wandering set, periodic orbits and hyperbolicity}\label{unifh}

Throughout,  let $M$ be a compact, connected Riemannian manifold without boundary.
Given a $C^1$-vector field $X\in \mathfrak{X}^1(M)$, we often denote by $(X_t)_{t\in \mathbb R}$ the $C^1$-flow 
generated by $X$ and let $\Omega(X)$ denote its non-wandering set.  Given a compact $(X_t)_{t\in \mathbb R}$-invariant set $\Lambda\subset M$, 
we say that $\Lambda$ is a \emph{hyperbolic set}  for $(X_t)_{t\in \mathbb R}$ if there exists a 
$DX_t$-invariant splitting $T_\Lambda M = E^s \oplus E^0 \oplus E^u$ so that:
(a) $E^0$ is one dimensional and generated by the vector field $X$, and (b) there are constants 
$C>0$ and $\lambda\in (0,1)$ so that 
$$
\| D X_t (x)\mid_{E^s_x} \| \le C \lambda^t 
	\quad\text{and}\quad
	\| (DX_t (x)\mid_{E^u_x})^{-1} \| \le C \lambda^t 
$$
for every $x\in \Lambda$ and $t\ge 0$. 
The flow $(X_t)_{t\in \mathbb R}$ is \emph{Anosov} if $\Lambda=M$ is a hyperbolic set.
A point $p \in M$ is: (i) a \emph{singularity} if $X_t(p) = p$ for all $t \in \mathbb{R}$, and (ii) \emph{periodic} if 
there exists $t > 0$ such that $X_t(p) = p$ and $\pi(p):=\inf\{t > 0: \, X_t(p) = p\} > 0$.
We denote by $Sing((X_t)_t)$ (or $Sing(X)$) the set of singularities, by $\text{Per}((X_t)_t)$ (or $\text{Per}(X)$) the set of periodic orbits and by $Crit((X_t)_t):=Sing((X_t)_t) \cup \text{Per}((X_t)_t)$ the set of all critical elements for the flow 
$(X_t)_t$. 
Every non-singular point is called \emph{regular}.
A flow $(X_t)_t$ is called \emph{Axiom A} if $\Omega(X)$ is a uniformly hyperbolic set 
and $\overline{\text{Crit}((X_t)_t)}=\Omega(X)$. 
A periodic point $p \in M$ is \emph{hyperbolic} if its orbit $\gamma_p:= \{X_t(p): t \in \mathbb{R}\}$
is a hyperbolic set. A singularity $\sigma \in M$ is \emph{hyperbolic} if $DX(\sigma)$ has no purely imaginary 
eigenvalues. Denote by $W^s(\gamma_p)$ (resp. $W^u(\gamma_p)$) the usual stable
(resp. unstable) manifolds of the hyperbolic periodic orbit $\gamma_p$. The stable and unstable
manifolds for a singularity $\sigma$ are denoted similarly by $W^s(\sigma)$ and $W^u(\sigma)$, respectivelly
(we refer the reader to \cite[Subsection~2.1]{vitorzeze} for the classical definitions of stable and unstable manifolds). 
The \emph{homoclinic class} associated to a hyperbolic critical element $z$ is the compact invariant 
set $H(z)=\overline{W^s(\gamma_z)\pitchfork W^u(\gamma_z)}$.
Finally, we say that a flow $(X_t)_t$ is \emph{transitive} if there exists $x\in M$
so that $(X_t(x))_{t\in \mathbb R_+}$ is dense in $M$.
We refer to \cite{KH, palisdemelo} for more details on uniform hyperbolicity and homoclinic classes.

\subsubsection{Commuting flows and centralizers}

Given $r \geq 0$, let $\mathcal F^r(M)$ denote the space of $C^r$-flows on $M$.
We say that the flows $(X_t)_{t\in \mathbb R}$ and $(Y_t)_{t\in \mathbb R}$ \emph{commute} if
$Y_s \circ X_t = X_t \circ  Y_s$ for all $s, t\in \mathbb R$.
Given a flow $(X_t)_{t\in\mathbb R} \in \mathcal F^r(M)$, the \emph{centralizer of $X$} is the set
of $C^r$-flows that commute with $X$, that is, 
\begin{equation}\label{eq:defcentPhi}
\mathcal Z^r ((X_t)_t)
	=\{ (Y_s)_{s\in \mathbb R} \in \mathcal F^r(M) 
		\colon Y_s \circ X_t = X_t \circ  Y_s , \forall \, s, t\in \mathbb R\}.
\end{equation}
It is clear from the definition that all flows obtained as smooth reparametrizations of $(X_t)_t$ 
belong to $\mathcal Z^r ((X_t)_t)$. 
For that reason, the centralizer of the time-one map of a flow is never a discrete subgroup of the space of 
diffeomorphisms.
Given $r\ge 0$, we say a flow $(X_t)_t \in \cF^r(M)$ has \emph{quasi-trivial  centralizer} if for any $Y \in \mathcal{Z}^r(X)$  there exists a continuous function $h: M \rightarrow \mathbb{R}$ so that 
\begin{itemize}
\item[(i)] (orbit invariance) $h(x) = h(X_t(x))$ for every $(t, \, x) \in \mathbb{R} \times M$, and 
\item[(ii)] $Y_t(x) = X_{h(x)t} \, (x)$ for every $(t, \, x) \in \mathbb{R} \times M$.
\end{itemize}
In the case that the reparametrizations $h$ are necessarily constant then we say the centralizer is \emph{trivial}.
The previous notion has a dual formulation in terms of vector fields. 
Given $r\ge 1$ and $X\in \mathfrak{X}^r(M)$, one can define the \emph{centralizer of the vector field $X$} by
\begin{equation}\label{eq:defcentX}
\mathcal Z^r (X)
	=\{ Y \in \mathfrak{X}^r(M) 
		\colon [X,Y]=L_Y X= 0\},
\end{equation}
where $[X,Y]$ denotes the usual comutator of the vector fields $X$ and $Y$, and 
$L_Y X$ denotes the Lie derivative of the vector field $X$ along $Y$. 
We say that the vector field $X \in \mathfrak{X}^r(M)$ has \emph{quasi-trivial centralizer} if for any $Y \in \mathcal{Z}^r(X)$ there exists a continuous $h: M \to \mathbb R$ 
that is constant along the orbits of $(X_t)_t$ (i.e. is a first integral for the flow) and so that $Y = h \cdot X$. 
As before, the centralizer is \emph{trivial} if any such $h$ is necessarily constant.
Observe that any $C^1$ reparametrization $h: M \to \mathbb R$ satisfying $X(h)=0$ is constant along the
orbits of the flow and, hence, is a first integral for the flow. 
If $(X_t)_t$ is transitive then any first integral is necessarily constant. For that reason, any transitive flow 
with a quasi-trivial centralizer has trivial centralizer. Moreover, 
as $C^r$-generic vector fields do not admit  non-trivial first integrals \cite{mane}, the set of $C^1$-flows
with a quasi-trivial and non-trivial centralizer is meager.

\subsubsection{Hamiltonian vector fields and flows}\label{subsec:Hamilton}

Let $(M^{2n}, \, \omega)$ be a compact symplectic Riemannian manifold. Given $r\ge 1$ and 
a Hamiltonian $H \in C^{r+1}(M,\mathbb R)$, the Hamiltonian vector field $X_H\in \mathfrak{X}^{r}_\omega(M)$ is defined by $\omega(X_H(x),v)=DH(x) \,v$ for every $x\in M$ and $v\in T_x M$. In the case $r\ge 2$ the Hamiltonian vector field $X_H$ is at least $C^1$ and we denote by $(\varphi^H_t)_t$ the Hamiltonian flow generated by $X_H$.
We say that $K\in C^{r+1}(M,\mathbb R)$ is a \emph{first integral} (or an infinitesimal symmetry) for the Hamiltonian $H\in C^{r+1}(M,\mathbb R)$ 
if $K(\varphi^H_t(x))=K(x)$ for every $x\in M$ and $t\in \mathbb R$. A first integral $K$ for a Hamiltonian $H$ is characterized
by $\{K,H\}=0$, where the Poisson bracket $\{\cdot, \cdot\} : C^{r+1}(M,\mathbb R)\times C^{r+1}(M,\mathbb R) \to C^{r}(M,\mathbb R)$, defined by
\begin{equation}\label{def:Poissonbracket}
\{H,K\}(x) = \frac{d}{dt} H(\varphi^K_t)\mid_{t=0}, 
\end{equation}
measures the derivative of $H$ along the orbits of $(\varphi^K_t)_t$. 
Recall $\{H,K\}=-\{K,H\}$ and that the Poisson and Lie brackets for Hamiltonians are related by $[X_K, \, X_H] = X_{\{H, \, K\}}$ (see e.g. \cite{Lee}). 
Hence, the Poisson bracket  determines commutativity of Hamiltonians: $H,K \in C^r(M,\mathbb R)$ commute if and only if the Poisson bracket $\{H,K\}$ is locally constant (see e.g. \cite[Section~40]{Arnold}).  Nevertheless, two Hamiltonian flows may commute and not preserve 
level sets (see Example~\ref{commuteR2H}). 
On the one hand, Noether's  Theorem (see e.g. \cite{KH} or \cite[Theorem 18.27]{Lee}) establishes a bijective correspondence between the set of vector fields $Y \in \mathfrak{X}^r(M)$ such that $H(Y_t(x)) = H(x)$  and $Y_t^{*} \omega = \omega$
for any $t \in \mathbb{R}$ and $x \in M$ (also known as conserved quantities) and the set of first integrals for $X_H$, modulo addition by constants. On the other hand, $K$ is a first integral of $X_H$ iff $\{K, \, H\} = 0$ (see e.g. \cite[Theorem~2]{meyer}). 

These facts motivate the following definition.
Given $r\ge 1$, a Hamiltonian $H \in C^{r+1}(M, \mathbb R)$ with associated vector field $X_H \in \mathfrak{X}^{r}(M)$, the \emph{Hamiltonian $C^r$-centralizer} of $X_H$ is 
\begin{equation}\label{eq:defcentXH}
\mathcal{Z}^r_{\omega}(X_H) = \{X_K \in \mathfrak{X}^r_{\omega}(M): \, \{K, \, H\} = 0\}.
\end{equation}
Then, we say that the centralizer of a Hamiltonian flow  $(\varphi_t^H)_t$ is: (i) \emph{quasi-trivial} if 
for every $X_K \in \mathcal{Z}^r_{\omega}(X_H)$ there exists a continuous map $h$ such that $X_H(x)= h(x) X_K(x)$ for every $x\in M$, and (ii) \emph{trivial} on the connected component 
$\mathcal E_{H,e} \subset H^{-1}(e)$ if for every $X_K \in \mathcal{Z}^r_{\omega}(X_H)$ there exists $c\in \mathbb R$
such that $X_K(x)=c X_H(x)$ for every $x\in \mathcal E_{H,e}$.

\subsection{Statement of the main results}\label{subsec:main}

First we will state the criterion for the triviality of the centralizer of vector fields and flows. 

\begin{maintheorem}\label{thm:criteriumH} 
Let $M$ be a compact Riemannian manifold and  and let $(X_t)_t$ be the flow generated by $X \in \mathfrak{X}^r(M)$, 
for some integer $r\ge 1$. Assume that $\Lambda\subset M$ is a compact, $(X_t)_t$-invariant and transitive subset. 
If the set of hyperbolic periodic orbits of $(X_t)_t$ is dense in $\Lambda$ then $\mathcal{Z}^1(X\mid_\Lambda)$ is trivial:
if $Y \in \mathcal{Z}^1(X)$ then the flow $(Y_t)_t$  preserves $\Lambda$ and there exists $c \in \mathbb{R}$ 
such that $Y\!\mid_{\Lambda} = c X\!\mid_{\Lambda}$.
\end{maintheorem}

Theorem~\ref{thm:criteriumH} has a first application to non-uniformly hyperbolic flows. 
Assume that $(X_t)_t$ is the flow generated by $X\in \mathfrak{X}^{1}(M)$ and let $\mu$ be a non-atomic and 
$(X_t)_t$-invariant probability measure. It follows from
Oseledets theorem (\cite{O}) that for $\mu$-almost every $x$ there exists a $DX_t$-invariant decomposition 
$T_xM= E^0_x \oplus E^1_x \oplus E^2_x\oplus \dots \oplus E_{x}^{k(x)}$, called the \emph{Oseledets splitting}, where $E^0_x$ is the one-dimensional subspace generated by $X(x)$ and 
there are well defined real numbers
$$
\lambda_i(X,x):= \lim_{t\to\pm \infty} \frac1t \log \|DX_t(x) v_i\|,
	\qquad \forall v_i \in E^i_x\setminus \{\vec0\} \text{ and }
	 1\le i\le k(x)
$$
called the \emph{Lyapunov exponents} associated to the flow $(X_t)_t$ and $x$. It is well known that if $\mu$ is 
ergodic then the Lyapunov exponents are almost everywhere constant and are denoted by $\lambda_i(X,\mu)$, $1\le i \le k$.
Such an ergodic probability measure $\mu$ is \emph{hyperbolic} if 
$\lambda_i(X,\mu)\neq 0$
for every $1\le i \le k$.
Given a $C^{2}$-diffeomorphism $f$, the support of any non-atomic, $f$-invariant, ergodic and hyperbolic probability measure 
is transitive and contains a dense set of periodic orbits (cf.~\cite[Theorem 4.1]{Katok}). In the case of vector fields this result is a consequence of \cite{LY}.
We refer the reader to \cite{BP} for an excelent account on non-uniform hyperbolicity.
Then, Theorem~\ref{thm:criteriumH} has the following immediate consequence:

\begin{maincorollary}\label{cor:Pesin}
Let $M$ be a compact Riemannian manifold and let $(X_t)_t$ be the flow generated by $X \in \mathfrak{X}^r(M)$, 
for some integer $r\ge 2$. If $\mu$ is a non-atomic $(X_t)_t$-invariant, ergodic and hyperbolic probability measure then
$\mathcal{Z}^1(X\mid_{\supp \mu})$ is trivial.
\end{maincorollary}

We note that the assumptions of Theorem~\ref{thm:criteriumH} are also satisfied by homoclinic classes. 
Indeed, a homoclinic class $\Lambda$ is transitive and it follows from Birkhoff-Smale's theorem 
that hyperbolic periodic orbits are dense (see e.g. \cite{KH}).
Thus we also get the following immediate consequence.

\begin{maincorollary}\label{cor:homoclinic}
Assume that $r\ge 1$, $X\in \mathfrak{X}^r(M)$, $p\in M$ is a hyperbolic critical element for $X$ and $\Lambda=H(p)$ 
is a homoclinic class for $X$. Then $\mathcal{Z}^1(X\mid_\Lambda)$ is trivial.
\end{maincorollary}

Corollary~\ref{cor:homoclinic},  which applies e.g. to hyperbolic basic pieces of $C^r$-flows, will be useful to study the triviality of the centralizer 
for non-hyperbolic homoclinic classes, including volume preserving  and Hamiltonian vector fields, and vector fields that exhibit Lorenz attractors 
(see Section~\ref{sec:prelim} for definitions).

\begin{remark}\label{rmkcrit}
Theorem~\ref{thm:criteriumH} and Corollary~\ref{cor:homoclinic} do not admit a counterpart for discrete time dynamics. 
Indeed, there are linear Anosov automorphisms on $\mathbb T^n$ (hence transitive, with a dense set of 
periodic orbits, all of them hyperbolic) that do not have trivial centralizer (see e.g.~\cite{Plykin}).
Nevertheless, any Anosov flow (e.g. suspensions of an Anosov diffeomorphisms)
satisfies the hypothesis of Theorem~\ref{thm:criteriumH} and, consequently, has trivial centralizer.
\end{remark}

In order to state some consequences of Theorem~\ref{thm:criteriumH} and Corollary~\ref{cor:homoclinic} we recall that a subset $R$ of a topological 
space $E$ is called \emph{Baire residual} if it contains a countable intersection of open and dense sets, and it is called \emph{meager} if it is the complementary set of a Baire residual subset. It follows from Baire category theorem that if $E$ is a Baire space then all residual subsets are dense in $E$. Given $r\ge1$,
let $\mathfrak{X}^r_\mu(M)$ denote the space of $C^r$-volume preserving vector fields on $M$. 

\begin{maincorollary}\label{ks}
Let $r\ge 1$ and let $M$ be a compact and connected Riemannian manifold. 
There exists a $C^r$-residual subset $\mathcal R \subset \mathfrak{X}^r(M)$ 
such that if $X\in \mathcal R$, $Y\in \mathcal{Z}^1(X)$ and
$\Lambda\subset \overline{Per(X)}$ is a transitive invariant set then there exists $c\in \mathbb R$ so that $Y=cX$ 
on $\Lambda$.
Moreover, $C^1$-generic vector fields in $\mathfrak{X}^1_\mu(M)$ have trivial $C^1$-centralizer.
\end{maincorollary}

The second assertion in Corollary~\ref{ks} is the counterpart of the results in \cite{BCW2} for $C^1$-volume preserving 
vector fields. 

Now we draw our attention to the existence of $C^1$-open sets of vector fields with trivial centralizer.
Kato and Morimoto~\cite{Kato} used the notion of topological stability to prove that all $C^1$-Anosov flows have 
quasi-trivial centralizer (hence transitive Anosov flows have trivial centralizer). 
In ~\cite{Sad2}, Sad used a linearization technique 
to prove that the centralizer is trivial for a $C^\infty$-open and dense set of $C^\infty$-Axiom A flows with the strong transversality condition. Theorem~\ref{thm:criteriumH} can also be used to exhibit new $C^1$-open sets of vector fields
that have non-uniformly hyperbolic attractors on which the centralizer is trivial. 
These are open sets of vector fields that
exhibit Lorenz attractors (we refer the reader to \cite{vitorzeze} for definitions and a large account on Lorenz attractors).   
More precisely we obtain the following:

\begin{maincorollary}\label{cor:Lorenz}
Let $\mathcal U \subset \mathfrak{X}^1(\mathbb R^3)$ be a $C^1$-open set of vector fields and let $V \subset \mathbb R^3$ be an open elipsoid containing the origin such that every $X\in \cU$ exhibits  a  Lorenz attractor 
$\Lambda_X = \bigcap_{t\ge 0} \overline{X_t(V)}$. 
Then $\mathcal Z^1(X\mid_{\Lambda_X})$ is trivial for every $X\in \mathcal U$: 
for any $Y \in \mathcal{Z}^1(X)$ the flow generated by $Y$ preserves the $(X_t)_t$-invariant set $\Lambda_X$ 
and there exists $c \in \mathbb{R}$ such that $Y\mid_{\Lambda_X} = c X\mid_{\Lambda_X}$.
\end{maincorollary}

The previous result is complementary to \cite[Theorem~A]{BRV} where the authors proved that a $\mathcal{C}^1$-open and $\mathcal{C}^{\infty}$-dense subset of vector fields $\mathfrak{X}^\infty(M^3)$ that exhibit Lorenz attractors have trivial centralizer on their topological basin of attraction.

\begin{remark}
The argument in the proof of Corollary~\ref{cor:Lorenz} consists of checking that every singular-hyperbolic attractor, at least for
three-dimensional flows, is a homoclinic class
(see e.g. Theorem 6.8 in \cite{vitorzeze}).
Hence, the triviality  of the centralizer stated in Corollary~\ref{cor:Lorenz} also holds for two other classes of Lorenz attractors. Indeed, Bautista and Rojas \cite{BR} proved that the contracting Lorenz attractors (also called Rovella attractors) are homoclinic classes, while
Metzger and Morales \cite{MM} proved that the multidimensional Lorenz attractors are also homoclinic classes (we refer the reader 
to \cite{BR,MM} for the definitions of the attractors).
\end{remark}

In what follows we discuss some applications of our result in the case of Hamiltonian vector fields.
In this setting, the triviality of the centralizer is not immediate from the denseness of periodic orbits whose periods are isolated. Indeed, given a Hamiltonian $H\in C^2(M,\mathbb R)$, any $K\in C^2(M,\mathbb R)$ such that $X_K\in \mathcal Z^1_\omega(X_H)$ is a first integral for the flow $(\varphi_t^H)_t$ (cf. Subsection~\ref{subsec:Hamilton}).
However is not clear \emph{a priori} that $K$ preserves the level sets $H^{-1}(e)$ and, even if this is the case, 
the first integrals $H$ and $K$ could be independent. 
Here we prove that 
$C^2$-generic Hamiltonians have trivial centralizer. More precisely: 

\begin{maincorollary}\label{hamil1} 
Let $(M^{2n},\omega)$ be a compact symplectic Riemannian manifold. 
If $n \geq 2$ then there exists a residual subset $\mathcal{R} \subset C^2(M, \mathbb R)$ such that 
the following holds: for every $H\in \mathcal R$ there 
exists a residual and full Lebesgue measure subset
$\mathfrak R_H \subset H(M)$ of energies such that if 
$e \in \mathfrak R_H$ then the centralizer of the
Hamiltonian flow  $(\varphi_t^H)_t$ is trivial on each connected component  $\mathcal E_{H,e} \subset H^{-1}(e)$. 
\end{maincorollary}

As Anosov energy levels persist by $C^1$-perturbations of the vector field and volume preserving
Anosov flows are transitive, the following is an immediate consequence from Theorem~\ref{thm:criteriumH}:

\begin{maincorollary}\label{hamil2} Let $(M^{2n},\omega)$ be a compact symplectic Riemannian manifold with $n\ge 2$,
let $H_0 \in C^r(M, \mathbb R)$ ($r\ge 2$) and let $e_0\in \mathbb R$ be such that the Hamiltonian flow 
$(\varphi_t^{H_0})_t$ restricted to a connected component  $\mathcal E_{H_0,e_0} \subset H_0^{-1}(e_0)$ is Anosov.
There is a $C^r$-open neighborhood $\mathcal U$ of $H_0$ and a open neighborhood $V\subset \mathbb R$ of $e_0$
such that for every $(H, e) \in \mathcal U \times V$ the flow $(\varphi^{H}_t)_t$ on has trivial centralizer on the connected component $\mathcal E_{ H,e} \subset \tilde H^{-1}(\tilde e)$ obtained as an analytic 
continuation of $\mathcal E_{H_0,e_0}$.
\end{maincorollary}

One should note that in the case of surfaces one cannot avoid the presence of elliptic islands 
and, in particular, the triviality of the centralizer cannot be expected to hold in general (cf. Example~\ref{centcentror2}). 

\smallskip 
We proceed
by stating  a result on the $C^0$-centralizer of continuous flows. 
Despite the fact that the proof of Theorem~\ref{thm:criteriumH} makes use of some topological arguments, the $C^1$-smoothness
of the vector field (and the hyperbolicity of a dense subset of periodic orbits) is used to assure that the invariant set 
$\Lambda$ is preserved by elements in the centralizer, and that hyperbolic periodic orbits are also preserved
(cf. Lemma~\ref{le:hip} and the argument following it). 
For general continuous flows, invariant sets may not be preserved by elements of the centralizer (cf. Example~\ref{centcentror2}). 
Combining arguments used in the proof of Theorem~\ref{thm:criteriumH} and in \cite{BRV} we overcome this difficulty and obtained 
the following result: 

\begin{maintheorem}\label{thm:B} 
Let $M$ be a compact and connected Riemannian manifold without boundary, and let $(X_t)_t$ be a continuous flow 
such that:
\begin{enumerate}
\item $(X_t)_t$ has at most finitely many equilibrium points, 
\item $\text{Per}((X_t)_t)$ is dense in $M$,
\item $(X_t)_t$ is transitive, and
\item for every $T>0$ the set of periodic orbits of $(X_t)_t$ with period smaller than $T$ 
form a subset of $M$ with finitely many disjoint closed curves.
\end{enumerate}
Then $\mathcal{Z}^0((X_t)_t)$ is trivial.
\end{maintheorem}

Finally, it is natural to ask whether 
our results 
extent to the setting of transitive and locally free 
$\mathbb R^d$-actions with a dense set of closed orbits. 
Although some open classes of $\mathbb R^d$-actions can be obtained as suspensions of  
$\mathbb Z^d$-actions (see e.g. \cite{BRV} and references therein), to the best of our knowledge, 
the existence of closed orbits is itself a wide open problem in this context with a single contribution in 
the case of $\mathbb R^2$ actions on $3$-manifolds~\cite{RW}.

\medskip
This paper is organized as follows. In Section~\ref{sec:examples} we provide some examples that illustrate
our main results. Section~\ref{sec:proofs} is devoted to the proof of the main results. Indeed, Theorem~\ref{thm:criteriumH}, 
its consequences (Corollaries~\ref{ks}, \ref{cor:Lorenz}, \ref{hamil1} and \ref{hamil2}) and Theorem~\ref{thm:B} are proven in
Subsection~\ref{sec:homoclinic}, Subsection~\ref{sec:applications} and Subsection~\ref{sec:Alinha}, respectively.

\section{Examples}\label{sec:examples}

In this section we give some examples 
that illustrate our results. In rough terms, the tubular flowbox theorem ensures that a smooth vector field is conjugated
to a constant vector field in small neighborhoods of regular orbits (see the precise statement in Example~\ref{centfluxotubular} below). 
In particular, the following simple example of constant vector fields describes the possible local symmetries of smooth vector fields.

\begin{example}\label{centfluxotubular}
Let $X \in \mathfrak{X}^r(\mathbb{R}^n)$ be the constant vector field $X(x) = (1, \, 0, \, \cdots, \, 0)$. If $Y \in \mathcal{Z}^r(X)$ is written in coordinates as $Y(x) = (Y_1(x), \, \cdots, \, Y_n(x))$ and the vector fields $X$ and $Y$ commute
then, 
in local coordinates,
\begin{align*}
0 = [X, \, Y] & = \sum\limits_{i = 1}^n \left[\sum\limits_{j = 1}^n \left(X_j \frac{\partial Y_i}{\partial x_j} - Y_j \frac{\partial X_i}{\partial x_j} \right)\right] \frac{\partial}{\partial x_i} \\
              & = \sum\limits_{i = 1}^n \frac{\partial Y_i}{\partial x_1} \frac{\partial}{\partial x_i}
               = \left(\frac{\partial Y_1}{\partial x_1}, \, \frac{\partial Y_2}{\partial x_1}, \, \cdots, \, \frac{\partial Y_n}{\partial x_1}\right).
\end{align*}
In other words, $X$ and $Y$ commute if and only if $\frac{\partial Y_i}{\partial x_1} = 0$ for all $1 \leq i \leq n$ or, equivalently, the vector field $Y$ does not depend on the $x_1$-coordinate, hence it is constant along
the orbits of the constant vector field.
\begin{figure}[hbt]
       \includegraphics[scale = 0.6]{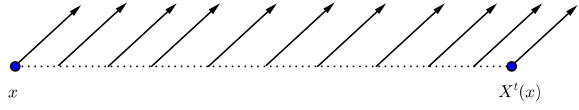}
	\caption{Local rigidity on flowbox charts caused by commutativity}
\end{figure}

If $M$ is a Riemannian manifold of dimension $n$, $Z \in \mathfrak{X}^r(M)$, $r \geq 1$, and $p \in M$ is a regular 
point for $Z$, the tubular flowbox theorem (cf. \cite{palisdemelo}) assures the existence of a neighborhood  $V_p$ of 
$x$ in $M$ and a $C^r$-diffeomorphism $g: V_p \rightarrow h(V_p) \subset \mathbb{R}^n$ so that 
$g_* Z=X\mid_{h(V_p)}$. If the vector field $\tilde Z\in \mathfrak{X}^r(M)$ commutes with $Z$ on $V_p$ 
then $0 = [Z,\tilde Z]  = [g_* Z,g_* \tilde Z] = [X,g_* \tilde Z]$. Here we used that Lie brackets are invariant by 
induced vector fields (see e.g. \cite[Corollary 8.31]{Lee}). The latter implies that the vector field $g_* \tilde Z$ does not
depend on the $x_1$-coordinate in the local coordinates of $g(V_p)\subset \mathbb R^n$.
\end{example}

Now we note that commuting Hamiltonian vector fields may not preserve level sets, which justifies the 
definition of the Hamiltonian centralizer.

\begin{example}\label{commuteR2H}
Consider the space $\mathbb R^2$ endowed with the usual symplectic form $\omega=dx \wedge dy$. Then, it is not hard to check that for any $C^{r+1}$ Hamiltonian $H: \mathbb R^2 \to \mathbb R$ the Hamiltonian vector field $X_H \in \mathfrak{X}^r(M)$ can be written as  $X_H(x,y)=\left(-\frac{\partial H(x,y)}{\partial y}, \frac{\partial H(x,y)}{\partial x}\right)$, 
for every $(x,y)\in \mathbb R^2$. It is not hard to check that the Hamiltonian flows $(\varphi^H_t)_t$, $(\varphi^K_t)_t$
generated by  $H, K: \mathbb R^2 \to \mathbb R$, $H(x,y)=x$ and $K(x,y)=y$, commute. However, the flow 
$(\varphi^K_t)_t$ does not preserve level sets of $(\varphi^H_t)_t$. In fact, the corresponding Poisson bracket is given by
$$
\{H, \, K\} 
   = \dfrac{\partial H}{\partial x} \, \dfrac{\partial K}{\partial y}  -  \dfrac{\partial H}{\partial y} \, \dfrac{\partial K}{\partial x} 
   =1 \neq 0.
$$
\end{example}

The triviality of the centralizer cannot be expected in general even for vector fields on surfaces. One can produce an example on the sphere $\mathbb S^2$ by a simple modification of the following vector field on the plane. 

\begin{example}\label{centcentror2} Consider the linear vector field in $\mathbb{R}^2$ given by $X(x, \, y) = (-y, \, x)$,
that generates the flow $(X_t)_{t\in \mathbb R}$ such that $X_t$ is the rotation of angle $t$ in $\mathbb R^2$. 
If $Z(x, \, y) = (a x + b y, \, c x + d y)$ is a linear vector field in $\mathbb{R}^2$ that commutes with $X$ then
\begin{align*}
0 = [X, \, Z] 
& = \Big(X_1 \dfrac{\partial Z_1}{\partial x} -Z_1 \dfrac{\partial X_1}{\partial x} 
	+ X_2 \dfrac{\partial Z_1}{\partial y} - Z_2 \dfrac{\partial X_1}{\partial y},  \\
&	\;\;\qquad  
	X_1 \dfrac{\partial Z_2}{\partial x} - Z_1 \dfrac{\partial X_2}{\partial x} 
	+ X_2 \dfrac{\partial Z_2}{\partial y} - Z_2 \dfrac{\partial X_2}{\partial y}  \Big)\\
              & = \big((b + c)x + (d-a)y, \;\; (d - a)x - (b+c)y \big),
\end{align*}
that is, $a = d$ and $c = -b$. Thus, the vector field $X$ commutes with the family of linear vector fields 
\begin{equation}\label{eq:Zab}
Z_{a,b}(x, \, y) = (a x + b y, \, -b x + a y)
\end{equation}
 for every $(x,y)\in \mathbb R^2$ ($a,b\in\mathbb R$). 
The centralizer of the linear vector field $X$ contains also nonlinear vector fields. For instance, it is not
hard to check that the vector field $Y\in \mathfrak{X}^{\infty}(\mathbb{R}^2)$ given by $Y(x, \, y) = \left(y + x(1 - x^2 - y^2), \; -x + y(1 - x^2 - y^2)\right)$ is such that $[X, \, Y] = 0$. 
\begin{figure}[htb]\label{fig1}
\begin{center}
       \includegraphics[scale = 0.4]{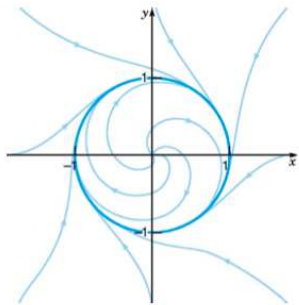}
	\caption{Representation of some solutions of the flow generated by the non-linear vector field $Y(x, \, y) = \left(y + x(1 - x^2 - y^2), \; -x + y(1 - x^2 - y^2)\right)$, $(x,y)\in \mathbb R^2$.}
\end{center}
\end{figure}
This shows that the one parameter group of rotations on $\mathbb{R}^2$, which correspond to the solutions of
the vector field $X$, are rotational symmetries for the solutions of the flow generated by $Y \in \mathcal{Z}^r(X)$
(see Figure~2).
In particular, it is clear that invariant sets may not be preserved by elements of the centralizer.
For instance, any annulus $\mathbb D$ centered at the origin is preserved by the flow $(X_t)_t$ while the 
flow generated by the vector field $Y(x, \, y) = (x, y)$ commutes with $(X_t)_t$ but does not preserve $\mathbb D$.
\end{example}

In the next example we consider the Hamiltonian centralizer of a linear center in $\mathbb R^2$.

\begin{example}
Consider the linear vector field $X \in \mathfrak{X}^1(\mathbb R^2)$ given in the previous example. 
It is a Hamiltonian
vector field: $X=X_H$ for 
$H\in C^{2}(\mathbb R^2,\mathbb R)$ given by $H(x,y) = \frac12 ({x^2 + y^2})$.
If $Y \in \mathcal{Z}^1_{\omega}(X)$, there exists $K \in C^2(\mathbb R^2, \, \mathbb{R})$ such that $Y := Y_K
=(-\frac{\partial K}{\partial y}, \frac{\partial K}{\partial x})$ 
and 
\begin{align*}
0 & = \{H, \, K\} 
   = \dfrac{\partial H}{\partial x} \, \dfrac{\partial K}{\partial y}  -  \dfrac{\partial H}{\partial y} \, \dfrac{\partial K}{\partial x} 
   = {x} \, \dfrac{\partial K}{\partial y}  -  {y} \, \dfrac{\partial K}{\partial x} 
   =  \left\langle (x, \, y), \,\, \left(\dfrac{\partial K}{\partial y}, \, -\dfrac{\partial K}{\partial x}\right) \right\rangle.
\end{align*}
Hence, the vectors $(x, \, y)$ and $\left(\dfrac{\partial K}{\partial y}, \, -\dfrac{\partial K}{\partial x}\right)$ are orthogonal.
As the vector field $X_H(x,y)=(-y,x)$ is also orthogonal to $(x,y)$ then the vector fields $X_H$ and $X_K$ are collinear
and, consequently, there exists 
$\kappa: \mathbb{R}_+ \rightarrow \mathbb{R}$ continuous 
such that $X_K(x,y) = \kappa(\|(x,y)\|) \cdot X_H(x,y)$ for all $(x,y) \in \mathbb{R}^2 \setminus\{(0,0)\}$. In particular, since the orbits of the flow $(\varphi^H_t)_t$ are circles centered at the origin, the orbit of each point $(x,y) \in \mathbb{R}^2 \setminus \{(0,0)\}$ by the flow $(\varphi^K_t)_t$  is either a singularity, in which case all points in a circle $\gamma$ centered at the origin passing through $(x,y)$ are also singularities, or is periodic and the orbit coincides with the circle $\gamma$. Moreover, 
$(0,0)$ is a singularity for the vector field $X_K$.
Hence, the Hamiltonian centralizer of $X_H$ is smaller than its centralizer in the space of all vector fields 
described in  Example~\ref{centcentror2} (since there are non-volume preserving flows in the family ~\eqref{eq:Zab}).
Nevertheless, the centralizer $\mathcal Z^1(X_H)$ is not trivial as, for instance, the Hamiltonian $K(x,y)=\frac12(x^2+y^2)^2$
satisfies $X_K(x,y)= \kappa(\|(x,y)\|) X_H(x,y)$, where $\kappa(z)=2z^2$ tends to zero as $z \to 0^+$, 
is continuous but it is not constant.
\end{example}

\section{Proofs} \label{sec:proofs}

\subsection{Proof of Theorem~\ref{thm:criteriumH}} \label{sec:homoclinic}

Fix an integer $r\ge 1$ and $X \in \mathfrak{X}^r(M)$, and take $Y \in \mathcal{Z}^1(X)$. 
Assume that the vector fields $X$ and $Y$ generate the flows $(X_t)_t$ and $(Y_s)_s$, respectivelly. By
the commutative relation  $X_t \circ Y_s = Y_s \circ X_t$ for all $t, \, s \in \mathbb{R}$ it follows
that  
the diffeomorphism $Y_s$ is a conjugation between the flow $(X_t)_t$ with itself, for every $s \in \mathbb{R}$.
Hence, if $p\in M$ is a periodic point of period $\pi(p)>0$ for $(X_t)_t$ ($\gamma_p$ denotes its orbit) 
and $s\in \mathbb R$ is arbitrary then
\begin{equation}\label{eq:periodicp}
X_{\pi(p)}(Y_s(p)) = Y_s(X_{\pi(p)} (p))  = Y_s(p). 
\end{equation}
We make use of the following simple lemma.

\begin{lemma}\label{le:hip}
Let $M$ be a compact Riemannian manifold and assume that the 
$C^1$ flow $(Y_s)_s$ on $M$ belongs to the $C^1$-centralizer of the $C^1$-flow $(X_t)_t$. 
Then, for every $s\in \mathbb R$ and every hyperbolic critical element $p \in \text{Crit}((X_t)_t)$ 
one has $Y_s (\gamma_p)=\gamma_p$.
\end{lemma}

\begin{proof}
Let $p\in \text{Per}((X_t)_t)$ be a hyperbolic periodic point of prime period $\pi(p)>0$ and take $s\neq 0$. 
Differentiating $X_{\pi(p)} \circ Y_{s} =Y_s \circ X_{\pi(p)}$  
we obtain
$$
D X_{\pi(p)} (Y_{s} (p)) \; DY_s(p) =DY_s ( X_{\pi(p)}(p)) \; DX_{\pi(p)}(p)
	= DY_s ( p) \; DX_{\pi(p)}(p)
$$
which shows, since $Y_s$ is a diffeomorphism, that $DX_{\pi(p)}(p)$ and $D X_{\pi(p)} (Y_{s} (p))$ are linearly 
conjugate.  This proves that $(Y_s(p))_{s\in \mathbb R}$ is a family of hyperbolic periodic points of period $\pi(p)$
for the flow $(X_t)_t$. Since hyperbolic periodic orbits of the same period are isolated we conclude that  $Y_s (\gamma_p)=\gamma_p$
for every $s\in \mathbb R$, which proves the lemma.
\end{proof}

The previous lemma assures that  every $Y\in \mathcal{Z}^1(M)$  preserves the set of hyperbolic
periodic orbits  in $\Lambda$ (and so it preserves $\Lambda$, if $\Lambda$ admits a dense subset of periodic orbits). 

Now, as vector fields $X$ and $Y$ are collinear on $\gamma_p$, for all
periodic points $p\in \Lambda$, we use that the periodic orbits of $(X_t)_t$ are dense in $M$ to derive that the vector fields 
$X$ and $Y$ are collinear on $\Lambda \setminus Sing(X)$. More precisely:

\begin{proposition}\label{prophconservativo} Let $r\ge 1$ and let $X \in \mathfrak{X}^r(M)$ generate a flow $(X_t)_t$
so that $\Lambda\subset M$ is a compact, $(X_t)_t$-invariant and transitive subset with a dense subset of hyperbolic
periodic orbits. For any $Y \in \mathcal{Z}^1(X)$ the flow $(Y_t)_t$ preserves $\Lambda$ and there exists a map 
$h: \Lambda \setminus Sing(X) \rightarrow \mathbb{R}$ such that $Y(x) = h(x) \cdot X(x)$ for all $x \in \Lambda \setminus Sing(X)$. Moreover, the following properties hold:
\begin{itemize}
\item[(a)] $h$ is uniquely defined,
\item[(b)] $h$ is constant along regular orbits of $X$, and
\item[(c)] $h$ is  continuous.
\end{itemize}
\end{proposition}

\begin{proof}
Fix $Y\in \mathcal Z^1(X)$ arbitrary. We know that $\Lambda$ is preserved by the flow $(Y_t)_t$.
First, suppose by contradiction that exists $x_0 \in \Lambda \setminus Sing(X)$ such that the vectors $X(x_0)$ and $Y(x_0)$ are linearly independent. Consider the continuous map $r: \Lambda  \setminus Sing(X) \rightarrow T_x M$ given by 
$$
r(x) = Y(x)- \frac{\left\langle Y(x), \, X(x) \right\rangle_x}{\left\langle X(x), \, X(x) \right\rangle_x} X(x),
$$ 
where $\left\langle \, \cdot \, , \, \cdot \, \right\rangle_x$ denotes the Riemannian metric of $M$ at $T_x M$. As $r(x_0) \neq 0$, by continuity there exists an open neighborhood $V_{x_0} \subset \Lambda$ of  $x_0$ such that $r(x) \neq 0$ for all $x \in V_{x_0}$. This cannot occur because $r(x) = 0$ for all 
$x \in \Lambda \cap Per((X_t)_t)$ and $Per((X_t)_t)$ is dense in $\Lambda$. Therefore we conclude that $r$ is identically zero on 
$\Lambda\setminus Sing(X)$ or, in other words, the vector fields $X$ and $Y$ are collinear on $\Lambda \setminus Sing(X)$.
In consequence, there exists a map $h: \Lambda \setminus Sing(X) \rightarrow \mathbb{R}$ such that $Y(x) = h(x) \cdot X(x)$ for all $x \in \Lambda \setminus Sing(X)$.

We proceed to prove properties (a), (b) and (c).
The conclusion of (a) is immediate. Indeed, if there are $h_1, \, h_2: \Lambda \setminus Sing(X) \rightarrow \mathbb{R}$ such that $h_1(x) \, X(x) = Y(x) = h_2(x) \, X(x)$, then $(h_1(x) - h_2(x)) \, X(x) = 0$ for all regular points $x\in \Lambda \setminus Sing(X)$. Thus $h_1(x) = h_2(x)$ for all $x \in \Lambda\setminus Sing(X)$ and $h$ is uniquely defined.

Let us prove now that $h$ is constant along regular orbits of $X$. Given a regular point $p$ for $X$, 
we have that 
\begin{equation}\label{eq:comute2Ya}
Y(X_t(p))=h(X_t(p)) X (X_t(p)) 
\qquad \text{for every} \; t\in \mathbb R.
\end{equation}
Then, differentiating the relation $Y_s\circ X_t = X_t \circ Y_s$
with respect to $s$ (at $s=0$) 
we conclude that
\begin{align}
Y ( X_t(p)) & = DX_t(p) Y(p) = DX_t(p) [h(p) X(p)] \nonumber \\
	& = h(p) DX_t(p)  X(p) 
	= h(p) X (X_t(p)) \label{eq:comute2Y}
\end{align}
for every $t\in \mathbb R$. 
Since $X_t(p)$ is a regular point for $X$, equations \eqref{eq:comute2Ya} and \eqref{eq:comute2Y}  show
that the reparametrization $h$ is constant along the orbits of $(X_t)_t$.  In other words, $h(x) = h(X_t(x))$ for all $x \in 
\Lambda \setminus Sing(X)$ and $t \in \mathbb{R}$, which proves (b).

We are left to prove the continuity of the reparametrization $h$. 
First we observe that $Y_t(x)=X_{h(x)t}(x)$ for all $x\in \Lambda \setminus Sing(X)$.
This is a simple consequence of the uniqueness of solution of the ordinary differential equation $u'=Y(u)$
and the fact that $h$ is constant along the orbits of $(X_t)_t$.

Assume by contradiction that $h$ is not continuous at $x\in \Lambda \setminus Sing(X)$. 
Then there exists 
a sequence $(x_n)_{n \in \mathbb{M}}$ convergent to $x$ 
such that 
$\inf_{n\in \mathbb N} |h(x_n) - h(x)|>0$. 
First observe that, as the vector fields $X$ and $Y$ are continuous on the compact $M$, 
$X(x) = \lim\limits_{n \rightarrow \infty} X(x_n) \neq 0$ and $|h(x_n)| = \|Y(x_n)\|/\|X(x_n)\|$, 
the sequence $(h(x_n))_n$ of real numbers is bounded.
Moreover, since $x$ is a regular point, the tubular flowbox theorem assures that  there exists $\vep_x>0$, an open neighborhood $V_x\subset M$ of $x$, and a $C^r$-diffeomorphism 
$g: V_x \rightarrow (-\vep_x,\vep_x) \times B(\vec 0,\vep) \subset \mathbb{R} \times \mathbb{R}^{n-1}$ so that 
$g_*X (y)= (1,0,\dots, 0)$ for every $y\in V_x$.  Choose any sequence $(s_n)_n$ of real numbers 
so that $(s_n, \, x_n) \rightarrow (s, \, x) \in (\mathbb{R}\setminus\{0\}) \times  (\Lambda\setminus Sing(X))$, 
that $|h(x) \cdot s|< \frac{\vep_x}2$ and $|h(x_n) \cdot s_n - h(x) \cdot s| \le \frac{\vep_x}2$ for all 
$n \in \mathbb{N}$. By construction, $\delta_0:= \inf_{n\in \mathbb N} |h(x_n) \cdot s_n - h(x) \cdot s|>0$.
Since the points $X_{h(x_n) \cdot s_n}(x)$ and $X_{h(x) \cdot s}(x)$ lie in the same piece of orbit 
in the tubular flowbox chart then there exists $\delta_1>0$ (depending on $\delta_0$) so that 
 $d(X_{h(x_n) \cdot s_n}(x), \; X_{h(x) \cdot s}(x)) \geq \delta_1>0$ for all $n\in\mathbb N$. 
Therefore,
\begin{align*}
\delta_1 & \leq d(X_{h(x_n) \cdot s_n}(x), \; X_{h(x) \cdot s}(x)) \\
         & \leq d(X_{h(x_n) \cdot s_n}(x), \; X_{h(x_n) \cdot s_n}(x_n)) + d(X_{h(x_n) \cdot s_n}(x_n), \; X_{h(x) \cdot s}(x))\\
         & = d(X_{h(x_n) \cdot s_n}(x), \; X_{h(x_n) \cdot s_n}(x_n)) + d(Y_{s_n}(x_n), \; Y_s(x))
\end{align*}
for all $n\in \mathbb N$. 
The second term in the right hand side above is clearly convergent to zero as $n$ tends to
infinity by the continuous dependence on initial conditions of the flow $(Y_s)_s$. 
Since $(h(x_n))_n$ is bounded, choosing a convergent subsequence $(h(x_{n_k}))_k$ we conclude that
$$
0< \delta_1 \leq d(X_{h(x_{n_k}) \cdot s_{n_k}}(x),\; X_{h(x_{n_k}) \cdot s_{n_k}}(x_{n_k})) + d(Y_{s_{n_k}}(x_{n_k}), \; Y_s(x)) \xrightarrow[]{k \to\infty} 0
$$
leading to a contradiction. This proves item (c) and completes the proof of the proposition.
\end{proof}

We are now in a position to complete the proof of Theorem \ref{thm:criteriumH}.
The previous argument shows that there exists a continuous map $h: \Lambda \setminus Sing(X) \rightarrow \mathbb{R}$ 
that is constant along orbits of $(X_t)_t$ such that $Y(x) = h(x) \cdot X(x)$ for all $x \in \Lambda \setminus Sing(X)$.
Using that the flow $(X_t)_t$ is transitive on $\Lambda$, there exists $x_0 \in \Lambda \setminus Sing(X)$ such that 
$\overline{\{X_t(x_0): \, t \in \mathbb{R}_+\}} = \Lambda$. As $h$ is constant along the orbits of $(X_t)_t$ we 
conclude that  
$h(x)= h(x_0)$ for every $x\in \Lambda \setminus Sing(X)$. Take $c=h(x_0)$. 
The later implies that the vector fields $Y$ and $c X$ coincide in an open and dense subset of $\Lambda$. Thus $Y=cX$,  which proves the triviality of the centralizer on $\Lambda$.
\hfill $\square$

\subsection{Applications} \label{sec:applications}

In this subsection we will consider the applications to volume preserving flows, Lorenz attractors and Hamiltonian flows.

\subsection*{Proof of Corollary~\ref{ks}}

Given $r\ge 1$, let $\mathcal R \subset \mathfrak{X}^r(M)$ denote the $C^r$-residual subset formed by Kupka-Smale
vector fields (see e.g. \cite{palisdemelo}). In particular, all critical elements of vector fields $X\in \mathcal R$ are 
hyperbolic  (hence the ones with the same period are isolated).
Thus, if $\Lambda \subset \overline{Per(X)}$ is a transitive invariant set for the flow generated by $X\in \mathcal R$,
the argument used in the first part of the proof of Theorem~\ref{thm:criteriumH} assures that all the periodic orbits in 
$\Lambda \cap \text{Per}(X)$ are 
preserved by every $Y\in \mathcal Z^1(X)$. In particular, the flow generated by $Y$ preserves the set $\Lambda$,
and the restriction of the vector field to the set $\Lambda$ has trivial centralizer, that is, 
there exists $c\in \mathbb R$ so that $Y=cX$ on $\Lambda$. This proves the first assertion of the corollary.

In the case of volume preserving vector fields one has that $\Omega(X)=M$ for every $X\in \mathfrak{X}^1_\mu(M)$, 
by Poincar\'e recurrence theorem. 
Moreover, there are 
$C^1$-generic sets $\mathcal R_1, \mathcal R_2, \mathcal R_3\subset \mathfrak{X}^1_\mu(M)$ 
so that every $X\in \mathcal R_1$ generates a topologically mixing (hence transitive) flow \cite{bessa-bc},
that periodic orbits of every vector field $X\in \mathcal R_2$ are dense in $M$  \cite[Theorem~11.1]{PughRob},
and 
that hyperbolic periodic orbits for vector fields $X\in \mathcal R_3$ are dense in the non-wandering set (cf. \cite{PughRob}, \cite{R0} and \cite[Proposition~3.1]{N}).
These facts, together with Theorem~\ref{thm:criteriumH} prove the second assertion in the corollary and complete its proof. 
\hfill $\square$

\subsection*{Proof of Corollary~\ref{cor:Lorenz}}

Let $\mathcal U \subset \mathfrak{X}^1(\mathbb R^3)$ be a $C^1$-open set of vector fields and an open elipsoid
$V \subset \mathbb R^3$ containing the origin such that every $X\in \cU$ exhibits  a  geometric Lorenz attractor 
$\Lambda_X = \bigcap_{t\ge 0} \overline{X_t(V)}$. For every $X\in \mathcal U$ there exists a periodic point 
$p_X \in \Lambda_X$ so that the Lorenz attractor $\Lambda_X$ coincides with the homoclinic class
$H(p_X):=\overline{W^{s}(p_X) \pitchfork W^{u}(p_X)}$ (cf. \cite[Proposition~3.17]{vitorzeze}). 
In particular, the restriction of the flow to the attractor is transitive and, by Birkhoff-Smale's theorem, admits a 
dense set of hyperbolic periodic orbits. The corollary is then a direct consequence of Theorem~\ref{thm:criteriumH}.
\hfill $\square$

\subsection*{Proof of Corollary~\ref{hamil1}}

Assume that $n \ge 2$.
First we will verify that vector fields of $C^2$-generic Hamiltonians and connected components of
generic energy levels satisfy the following: the restriction of the Hamiltonian flow to such connected components 
is transitive and admits a dense set of hyperbolic periodic orbits. Our starting point is the following result, which assures
that transitive level sets for $C^2$-generic Hamiltonians are abundant from both topological
and measure theoretical senses.

\begin{theorem}\label{teo25BFRV}\cite[Theorems~2 and 3]{BFRV}
Let $(M^{2n},\omega)$ be a compact symplectic Riemannian manifold. 
There is a residual set $\mathcal{R}_0$ in $C^2(M,\mathbb{R})$ such that for any $H\in\mathcal{R}$ there exists 
a Baire residual, full Lebesgue measure subset $\mathfrak R_H$ of regular energy levels in $H(M) \subset \mathbb R$ such that 
for every $e\in \mathfrak R_H$ the restriction of the flow $(\varphi^H_t)_t$ to every connected component $\cE_{H,e}$ of the
level set $H^{-1}(e)$ is transitive.
\end{theorem}

In order to prove the corollary we are left to show that hyperbolic periodic orbits are dense in typical level sets
of $C^2$-generic Hamiltonians. Although this argument is contained in \cite{BFRV} we include it here for completeness. 
Robinson \cite{R0} proved that there exists a $C^2$ residual subset 
 $\mathcal{R}_1\subset C^2(M,\mathbb{R})$ of 
 Hamiltonians such that for every $H\in \mathcal{R}_1$ all closed orbits are either hyperbolic or elliptic. 
 Note that, by Birkhoff fixed point theorem, the hyperbolic periodic points are dense on $M$ (cf. \cite{N} Proposition 3.1, Corollary 3.2 and \S6). Moreover,  by Theorem~11.5 in  \cite{PughRob}, Newhouse's result can be strengthened in a way that 
 there exists a $C^2$-generic subset $\mathcal R_2 \subset C^2(M, \mathbb R)$ such that the generic connected components of energy levels for $H\in \mathcal R_2$ contains a dense set of hyperbolic periodic orbits. 
The density of hyperbolic periodic orbits and the fact that the ones displaying the same period are isolated 
guarantees, as in Theorem \ref{thm:criteriumH}, that for every Hamiltonian $H\in \mathcal R_0 \cap \mathcal R_1\cap \mathcal R_2$, every $X_K \in \mathcal{Z}^1_{\omega}(X_H)$ and every connected component $\cE_{H,e}\subset H^{-1}(e)$ associated to an energy $e\in \mathfrak R_H$, there exists 
$h: \cE_{H,e} \rightarrow \mathbb{R}$ continuous and invariant along regular orbits of $X_H$,
such that 
$X_K(x) = h(x) \cdot X_H(x)$ for all $x \in \cE_{H,e}$. 
Transitivity at the hypersurface $\cE_{H,e}$ assures that the first integral $h$ is constant on $\cE_{H,e}$,
which completes the proof of  Corollary~\ref{hamil1}. \hfill $\square$

\subsection{Proof of Theorem~\ref{thm:B}} \label{sec:Alinha}

The proof explores arguments used in the proof of Theorem~\ref{thm:criteriumH} and in \cite{BRV}.
Given the continuous flow $(X_t)_t$ denote by $\cE=\cE((X_t)_t)$ and $\cP=\cP((X_t)_t)$ the set of equilibrium and 
period points for $(X_t)_t$, respectively. 
Given $[a,b]\subset \mathbb R$ and $x\in M$ denote $X_{[a,b]}(x)=\{X_t(x) \colon t\in [a,b]\}$.
Our starting point is the following lemma, which plays the role of the tubular flowbox theorem for continuous
flows.

\begin{lemma}\label{le:C0flowbox}
Let $(X_t)_t$ be a continuous flow.
For every $T>0$ and $x\in M$ there exist open neighborhoods $U$ of the compact set
$\{X_t(x) \colon 0\le t\le T\}$ and $V\subset M$ of $x$, so that 
$X_{[0,T]}(y)$ is contained in $U$ for every $y\in V$. 
\end{lemma} 

\begin{proof}
This is an immediate consequence of the uniform continuity of the restriction of the continuous flow $X: \mathbb R \times M \to M$ to the compact set $M \times [0,T]$.
\end{proof}

Given a flow $(Y_s)_s \in \mathcal{Z}^0((X_t)_t)$, 
it preserves the set of equilibrium points and periodic orbits of $(X_t)_t$ of every fixed period: if $T\ge 0$ and $X_T(p)=p$ then 
\begin{equation}\label{eq:perC0}
X_{T}(Y_s(p)) = Y_s(X_{T} (p))  = Y_s(p)
\quad\text{for every}\, s\in\mathbb R. 
\end{equation}
Given $T>0$ denote by $M_T$ the set of points in $M$ that are either equilibrium points or periodic points of period smaller or equal to $T$ for $(X_t)_t$. By assumptions (1) and (4), the set $M_T$ consists of a finite number of orbits, hence isolated. 
Moreover, since each map $Y_s$ is a conjugacy of the flow with itself it preserves singularities and periodic orbits
of a fixed period.  Hence,
$Y_s (\gamma_p)=\gamma_p$ for every $p\in Per((X_t)_t) \cup \cE$ and every $s\in \mathbb R$.
In particular $M\setminus M_T$ is a $(X_t)_t$-invariant and dense subset of $M$.
By assumption (1), the set $M_0:=M\setminus (\cE\cup \cP)$ is $(X_t)_t$-invariant and dense.
We now prove that the flow $(Y_s)_s$ preserves all orbits of $(X_t)_t$.

\begin{lemma}\label{le:1}
For every $x\in M$ and every $s\in \mathbb R$ the point $Y_s(x)$ belongs to the orbit of $x$ by $(X_t)_t$.
\end{lemma}

\begin{proof}
Arguing by contradiction, assume that there are $x_0 \in M_0$ 
and $t_0>0$ so that the point $y_0:=Y_{t_0}(x_0)$ does not belong to the 
orbit of $x_0$ by $(X_t)_t$. 
By uniform continuity of $Y_{t_0}$, for any $\vep>0$ there exists $\delta>0$ so that 
$Y_{t_0} (\overline{B(x_0,\delta)}) \subset B(y_0,\vep)$.
By assumption (4), equilibrium points for $(X_t)_t$ are preserved by the flow $(Y_s)_s$. Moreover,
the latter condition together with the uniform continuity of the flow at compact pieces of orbits
(cf. Lemma~\ref{le:C0flowbox})
implies that there exists an open neighborhood of the compact set 
$$
\{Y_s(x) \colon (x,s)\in \overline{B(x_0,\delta)} \times  [0,t_0]  \}
$$
formed by regular points for the flow $(X_t)_t$. 

We now use assumptions (2) and (4). 
Since periodic points are dense in $M$ and the set of periodic orbits of a bounded period are isolated, one can
pick a sequence $(x_n)_n$ of periodic points for $(X_t)_t$ so that $\pi(x_n)\gg t_0$ for every $n$ and $d(x_n, x_0) \to 0$. 
Since each $x_n$ is a periodic point, 
and $(Y_s)_s$ preserves all periodic orbits of $(X_t)_t$, there exists
$t_n \in \mathbb R$ 
such that 
\begin{equation}\label{eq2xy01}
X_{[0,t_n]}(x_{n})  =  Y_{[0,t_0]}(x_n)  
\end{equation} 
(see Figure~3 below).
In particular $X_{t_n}(x_n)=Y_{t_0}(x_n) \in B(y_0,\vep)$ for every large $n$ and
\begin{equation}\label{eq2xy}
X_{t_n}(x_n)=Y_{t_0}(x_n)  \to Y_{t_0}(x_0)=y_0 
	\quad\text{as}\; n\to\infty. 
\end{equation} 
\begin{figure}[!h]
\includegraphics[scale=.6]{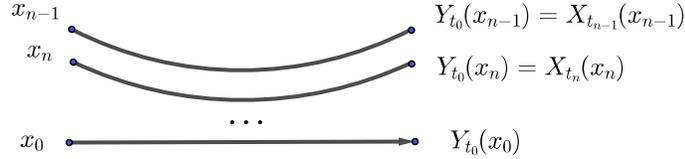}
\caption{Uniform convergence of the compact pieces of orbits $Y_{[0,t_0]}(x_n)$}
\end{figure}

Assume without loss of generality that $t_n\ge 0$ for all $n\ge 1$.
If the sequence $(t_n)_n$ admits a 
subsequence $(t_{n_k})_k$ convergent  to some $t_*\in \mathbb R_+$, the continuity 
of the flows and their time-t maps implies
$$
d(X_{t_{n_k}}(x_{n_k}), Y_{t_0}(x_{0})) = d(Y_{t_0}(x_{n_k}), Y_{t_0}(x_0))  \to 0
\quad \text{as $k\to\infty$}
$$
and, on the other hand,
$$
d(X_{t_{n_k}}(x_{n_k}), Y_{t_0}(x_{n_k}))  \to d(X_{t_*}(x_0), Y_{t_0}(x_0))
\quad \text{as $k\to\infty$.}
$$
This proves that $Y_{t_0}(x_0)=X_{t_*}(x_0)$, which is a contradiction with the choice of $x_0$.

Alternatively, the sequence $(t_n)_n$ tends to infinity as $n\to\infty$.
Fix an arbitrary $\ell>0$ and note that $\ell \ll t_n \le \pi(x_n)$ for every large $n\ge 1$. 
By continuity of the flow on compact trajectories we get
\begin{equation}\label{eq:per1}
X_{[0,\ell]}(x_{n}) \to  X_{[0,\ell]}(x_0) 
	\quad \text{as} \; n\to\infty. 
\end{equation}
Since the flow $(Y_s)_s$ preserves the orbits of the periodic points for $(X_t)_t$,  by \eqref{eq2xy01} 
there exists a unique $s_n\ge 0$ satisfying
\begin{equation}\label{eq:per2}
X_{[0,\ell]}(x_{n})  =  Y_{[0,s_n]}(x_n).  
\end{equation}
Note that $\ell \ll t_n$ implies $s_n \le t_0$.
We may assume that there exists a subsequence $(x_{n_k})_k$ of periodic points so that 
$s_{n_k} \to s_*:=\inf_{k\ge 1} s_{n_k}$ as $k\to\infty$. 
There are still two cases to consider.
If $s_*>0$ then it follows from equalities \eqref{eq:per1} and ~\eqref{eq:per2} and the convergence
$
Y_{[0,s_{n_k}]}(x_{n_k}) \to Y_{[0,s_*]}(x_0)
	\; \text{as} \; k\to\infty
$ 
that $X_{[0,\ell]}(x_0) =Y_{[0,s_*]}(x_0)$. In other words, the orbit of $x_0$ by the flow $(Y_s)_s$ preserves the 
piece of orbit $X_{[0,\ell]}(x_0)$. If $s_*=0$ then equality~\eqref{eq:per2} and uniform continuity of the flow ensures that 
$$
X_{[0,\ell]}(x_{n_k}) = Y_{[0,s_{n_k}]}(x_{n_k})=Y([0,s_{n_k}] \times \{x_{n_k}\}) \to Y_0(x_0)=x_0
$$ 
as $k\to\infty$.
Together with ~\eqref{eq:per1}, this implies that $x_0 \in \cE$, which contradicts the choice $x_0\in M_0$.
This proves that the flow $(Y_s)_s$ preserves all orbits of $(X_t)_t$ and
completes the proof of the lemma.
\end{proof}

We need the following lemma, inspired by \cite{BRV}, which will play a role similar to Proposition~\ref{prophconservativo}.

\begin{lemma}\label{le:existtau}
Given $x\in M_0$ there exists a unique function $\tau(x,\cdot): \mathbb R \to \mathbb R$ so that:
\begin{enumerate}
\item $Y_t(x) = X_{\tau(x, t)}(x)$ for every $x \in M$ and every $t\in \mathbb R$, and
\item $\tau(x,\cdot) : \mathbb R \to \mathbb R$ is continuous for every $x\in M_0$.
\end{enumerate}

\end{lemma}

\begin{proof}
First, we prove that a reparametrization $\tau(x,\cdot)$ as in the statement of the lemma does exists.
As $x\in M_0=M\setminus (\cP \cup \cE)$ and $t\in \mathbb R$ the point $Y_t(x)$ belongs to the orbit of $x$ by $(X_t)_t$, hence there exists $\tau(x,t)\in \mathbb R$ such that $Y_t(x) = X_{\tau(x, t)}(x)$.
It remains to prove the continuity of  $\tau(x,\cdot)$.
Assume that $t_n \to t \in \mathbb R$.
As $x\notin \cE \cup \cP$ then for every $\zeta>0$ and every large $n\ge 1$ we
have that $Y_{t_n}(x)$ belongs to the compact piece of orbit $\{  Y_s(x)\colon s\in [t-\zeta,t+\zeta]\}$
containing $x$ and 
\begin{align*}
X_{\tau(x, \, t_n)}(x) = Y_{t_n}(x)  \to Y_{t}(x) = X_{\tau(x, \, t)}(x), 
\end{align*}
from which we conclude that $\tau(x, \, t_n) \to \tau(x, \, t)$. This proves the continuity of $\tau(x,\cdot)$ 
as claimed.
It remains to prove the uniqueness of the reparameterization.
Fix $x\in M_0$ and assume there are $\tau_i(x,\cdot): \mathbb R \to \mathbb R$ ($i=1,2$) satisfying  (1)-(2) above. 
The equality $X_{\tau_1(x,t)}(x)=X_{\tau_2(x,t)}(x)$ for every $t\in \mathbb R$ ensures that 
\begin{equation}\label{eq:circular}
X_{\tau_1(x,t)-\tau_2(x,t)}(x)=x
\quad\text{for every }\; x\in \mathbb R.
\end{equation}
Since $x\in M_0$ then $\tau_1(x,t)-\tau_2(x,t)=0$ for all $t\in \mathbb R$. 
This proves the uniqueness of the reparameterization and concludes the proof of the lemma.

\end{proof}

\begin{lemma}
There exists a continuous map $h: M \setminus (\cE\cup\cP) \rightarrow \mathbb{R}$, constant along regular orbits of 
$(X_t)_t$, such that $Y_t(x) = X_{h(x) t}(x)$ for every $x \in M \setminus (\cE\cup\cP)$ and every $t\in \mathbb R$. \end{lemma}

\begin{proof}
Let $\tau$ be given by Lemma~\ref{le:existtau}.
We claim that there exists a function $h : M_0 \to \mathbb R$ so that 
$h(X_t(x))=h(x)$ and $\tau(x,t)=h(x) t$  
for every $x\in M_0$ and $t\in \mathbb R$.
Indeed, by Lemma~\ref{le:existtau}, given $x\in M_0$
and $s,t\in \mathbb R$,
\begin{align}
X_{t + \tau(x,s)} (x) & = X_{t} (X_{\tau(x,s)} (x)) = X_{t} (Y_s(x)) \nonumber \\
                              & = Y_{s} (X_t(x)) = X_{\tau(X_t(x),s)} (X_t(x)) = X_{\tau(X_t(x),s)+t} (x). \label{eq:comm}
\end{align}
By uniqueness of the reparameterization, the latter proves that $x\mapsto \tau(x,s)$ is constant along 
the orbits of $(X_t)_t$ for every $s\in \mathbb R$. Since $Y_s(x)$ belongs to the orbit of $x$ by $(X_t)_t$, we conclude that 
\begin{align*}
X_{\tau(x,t+s)} (x)  = Y_{t+s}(x) 
                                = Y_t ( Y_s(x)) 
                                = X_{\tau(Y_s(x),t)} (X_{\tau(x,s)}(x))   
                                = X_{\tau(x,t)+\tau(x,s)}(x)
\end{align*}
and so $\tau(x,t+s)=\tau(x,t)+\tau(x,s)$ for all $x\in M_0$ and $s,t \in \mathbb{R}$. {Since $\tau(x,\cdot)$ is continuous} (recall item (2) at Lemma~\ref{le:existtau}) then it is linear, and there exists $h(x)\in \mathbb R$ so that $\tau(x,t) = h(x) t$ for all $x \in M_0$. 

We are left to prove the continuity of $h$.
Assume by contradiction that $h: M_0 \rightarrow \mathbb{R}$  is not continuous.
Take $(x_n)_n \in M_0$ so that $(x_n)_n$ is convergent to 
$x\in M_0$ but there exists $\delta>0$ and a subsequence $(x_{n_k})_k$ such that $|h(x_{n_k})-h(x)|\ge \delta>0$ for all $k$.
There are two cases to consider:

\medskip
\noindent (i) \emph{there exists a subsequence, which we denote by  $(x_{n_k})_k$
for simplicity, so that $(h(x_{n_k}))_k$ converges to some $h\neq h(x)$;}
\medskip

\noindent Note that
 $X_{h(x)t}(x_{n_k})\to  X_{h(x)t}(x) = Y_t(x)$ and  
\begin{equation}\label{XYh}
Y_t( x_{n_k}) = X_{h(x_{n_k}) t}(x_{n_k}) 
= X_{[h(x_{n_k})-h(x)] t}(X_{h(x)t}(x_{n_k})). 
\end{equation}
Taking the limit as $k\to\infty$ in ~\eqref{XYh}
we conclude that
$$
Y_t( x) = X_{[h-h(x)] t}(Y_t( x)) 
	\quad\text{for every $t\in \mathbb R$}.
$$
As $h(x)\neq h$ the later implies that $Y_t( x) \in \cP(x)$ which contradicts the fact that $(Y_s)_s$ preserves the 
orbit of $x$ by $(X_t)_t$.

\medskip
\noindent (ii) \emph{$|h(x_n)|$ tends to infinity as $n\to\infty$;}
\medskip

\noindent Assume that there exists a subsequence $(x_{n_k})_k$ so that $\lim_{k\to\infty} h(x_{n_k})=+\infty$ (the other case is analogous) and fix $t>0$.
Given $\ell>0$ arbitrary we have $h(x_{n_k})t \ge h(x) t+\ell$ provided that $k\ge 1$ large.
By continuity of the flows $(X_t)_t$ and $(Y_s)_s$ on compact pieces of orbits (Lemma~\ref{le:C0flowbox}) 
we have that 
$$
X_{[0,\ell]}(X_{h(x_{n_k})t}(x_{n_k})) = X_{[0,\ell]}(Y_{t}(x_{n_k})) \to X_{[0,\ell]}(Y_t(x))
$$
and also
$
Y_{[0,t]}(x_{n_k}) \to Y_{[0,t]}(x).
$
Taking the limit as $k\to\infty$ in  ~\eqref{XYh} we conclude that the 
compact piece of orbit $X_{[0,(h(x_{n_k})-h(x))t]}(X_{h(x)t}(x_{n_k}))$ both converge to $Y_t(x)$, hence
obtaining that, on the one hand,
$$
X_{[0,\ell]}(X_{h(x)t}(x_{n_k})) \to X_{[0,\ell]}(X_{h(x)t}(x)) =Y_t(x) 
	\quad\text{as $k\to\infty$,}
$$
while
$$
X_{[0,\ell]}(X_{h(x_{n_k})t}(x_{n_k})) = X_{[0,\ell]}(Y_{t}(x_{n_k})) \to X_{[0,\ell]}(Y_t(x))
	\quad\text{as $k\to\infty$}.
$$
Altogether we conclude that $X_\ell(Y_t(x))=Y_t(x)$ for every $\ell>0$, which implies that
$Y_t(x)$ is a equilibrium point for $(X_t)_t$. Since the latter contradicts the fact that $(Y_s)_s$ preserves $M_0$ we conclude that $h$ is continuous, completing the proof of the lemma.
\end{proof}

We are in a position to complete the proof of Theorem~\ref{thm:B}. 
Indeed, since $(X_t)_t$ is transitive and
 $h: M_0 \rightarrow \mathbb{R}$ is constant along orbits of $(X_t)_t$ we conclude that
$h(x)=c$ for every $x\in M_0$. Therefore $Y_t(x) = X_{ c t}(x)$ for every $t\in \mathbb R$ and every $x$ on a dense subset of $M$. 
Since the flows $(X_t)_t$ and $(Y_s)_s$ are continuous then we conclude that $Y_t=X_{ct}$ for every $t\in \mathbb R$. This proves the triviality of the $C^0$-centralizer 
of $(X_t)_t$, as desired.

\vspace{.4cm}
\subsubsection*{Acknowledgements:} 
This work is part of the first author's PhD thesis at UFBA. The second author was partially supported by 
CNPq-Brazil. The authors are grateful to A. Arbieto and C. Maquera for valuable comments, and
indebted to the anonymous referee for the careful reading and a number of suggestions that helped to 
improved the manuscript.


\end{document}